\documentclass[a4paper, reqno, 10pt]{amsart}
\usepackage{color}
\usepackage{comment}
\usepackage{amsmath, amsfonts, amssymb, amsthm, graphics, graphicx}
\usepackage{hyperref}

\theoremstyle{plain}
\newtheorem{thm}{Theorem}[section]
\newtheorem{prop}[thm]{Proposition}
\newtheorem{lemma}[thm]{Lemma}

\theoremstyle{definition}
\newtheorem{remark}[thm]{Remark}

\providecommand{\aut}{\mathop{\rm Aut \,}\nolimits}
\providecommand{\sym}{\mathop{\rm Sym \,}\nolimits}
\providecommand{\soc}{\mathop{\rm soc}\nolimits}
\providecommand{\Wr}{\mathop{\rm Wr}\nolimits}

\providecommand{\Z}{\mathbb{Z}}

\renewcommand{\\}{\vspace{3mm}}

\usepackage[margin=30mm, a4paper]{geometry}

\title{\bf A classification of primitive permutation groups with finite stabilizers}
\author{\bf Simon M. Smith}
\email{sismith@citytech.cuny.edu}
\address{Department of Mathematics, NYC College of Technology, City University of New York (CUNY), New York, NY, USA}
\subjclass[2010]{20B07, 20B15 (primary), and 20E28 (secondary)}
\date{\today}
\begin{document}
\maketitle

\markboth{\textsc{Primitive groups with finite stabilizers}}{\textsc{Primitive groups with finite stabilizers}}

% Dedication
\begin{center} {\em
Dedicated to the memory of Stephen G.~Small.
} \end{center}

\begin{abstract}
We classify all infinite primitive permutation groups possessing a finite point stabilizer, thus extending the seminal Aschbacher--O'Nan--Scott Theorem to all primitive permutation groups with finite point stabilizers.
\end{abstract}

%
% --- Introduction
%
\section{Introduction}

\let\thefootnote\relax\footnotetext{The publisher requires the inclusion of the following statement: ``This manuscript has been accepted for publication in Journal of Algebra. The manuscript will undergo copyediting, typesetting, and review of the resulting proof before it is published in its final form. Please note that during the production process errors may be discovered which could affect the content, and all disclaimers that apply to the journal apply to this manuscript.''}

Recall that a transitive permutation group $G$ on a set $\Omega$ is primitive if the only $G$-invariant partitions are $\{\{\alpha\} : \alpha \in \Omega\}$ and $\{\Omega\}$. In the finite case, these groups are the fundamental actions from which all permutation groups are constituted.

The finite primitive permutation groups were classified by the famous Aschbacher--O'Nan--Scott Theorem, first stated independently by O'Nan and Scott. Scott's initial statement (\cite{Scott:Onan_Scott}) omitted the class of twisted wreath products; an extended and corrected version of the theorem appears in \cite{aschbacher_scott} and \cite{kovacs_onan_scott_correction}. The Aschbacher--O'Nan--Scott Theorem describes in detail the structure of finite primitive permutation groups in terms of finite simple groups. A modern statement of the theorem with a self-contained proof can be found in \cite{liebeck&praeger&saxl:finite_onan_scott}.\\

Primitive permutation groups with finite point stabilizers are precisely those primitive groups whose subdegrees are bounded above by a finite cardinal (\cite{schlichting}, \cite{bergman_lenstra}). This class of groups also includes all infinite primitive permutation groups that act regularly on some finite self-paired suborbit (see \cite[Problem 7.51]{kn}). Our main result is Theorem~\ref{thm:classification_of_finite_stabilizer} below; in conjunction with the Aschbacher--O'Nan--Scott Theorem, this yields a satisfying classification of all primitive permutation groups with finite point stabilizers, describing in detail their structure in terms of finitely generated simple groups.

\begin{thm} \label{thm:classification_of_finite_stabilizer}

If $G\leq \sym(\Omega)$ is an infinite primitive permutation group with a finite point stabilizer $G_\alpha$, then $G$ is finitely generated by elements of finite order and possesses a unique (non-trivial) minimal normal subgroup $M$; there exists an infinite, nonabelian, finitely generated simple group $K$ such that $M = K_1 \times \cdots \times K_m$, where $m\geq 1$ is finite and $K_i \cong K$ for $1 \leq i \leq m$; the stabilizer $G_\alpha$ acts transitively on the components $K_1, \ldots, K_m$ of $M$ by conjugation; and $G$ falls into precisely one of the following categories:

\begin{enumerate}
\item \label{classification_of_finite_stabilizer:split_extension_case} 
$M$ is simple and acts regularly on $\Omega$, and $G$ is equal to the split extension $M.G_\alpha$ for some $\alpha \in \Omega$, with no non-identity element of $G_\alpha$ inducing an inner automorphism of $M$;

\item \label{classification_of_finite_stabilizer:almost_simple_case} $M$ is simple, and acts non-regularly on $\Omega$, with $M$ of finite index in $G$ and $M \leq G \leq \aut(M)$;

\item \label{classification_of_finite_stabilizer:wreath_product_case}
$M$ is non-simple.  In this case $m>1$, and $G$ is permutation isomorphic to a subgroup of the wreath product $H \Wr_\Delta \sym(\Delta)$ acting via the product action on $\Gamma^m$, where $\Delta = \{1, \ldots, m\}$, $\Gamma$ is some infinite set and $H \leq \sym(\Gamma)$ is an infinite primitive group with a finite point stabilizer. Here $K$ is the unique minimal normal subgroup of $H$. Moreover, if $M$ is regular, then $H$ is of type~{\normalfont (\ref{classification_of_finite_stabilizer:split_extension_case})} and if $M$ is non-regular then
 $H$ is of type {\normalfont (\ref{classification_of_finite_stabilizer:almost_simple_case})}. 
\end{enumerate}
\end{thm}

For each type (\ref{classification_of_finite_stabilizer:split_extension_case}), (\ref{classification_of_finite_stabilizer:almost_simple_case}) and (\ref{classification_of_finite_stabilizer:wreath_product_case}) there exist examples of infinite primitive permutation groups with finite point stabilizers. We present these in Section~\ref{section:examples}.

For permutation groups which lie in classes (\ref{classification_of_finite_stabilizer:split_extension_case}) and (\ref{classification_of_finite_stabilizer:wreath_product_case}) there are known conditions which guarantee primitivity (Proposition~\ref{prop:G_prim_iff_Ga_no_normalising} and the paragraph immediately following it, and Lemma~\ref{lem:dixon_mortimer}).

For any group $G$ of type (\ref{classification_of_finite_stabilizer:wreath_product_case}), an explicit permutation embedding of $G$ into $H \Wr \sym(\Delta)$ is described in Lemma~\ref{lem:product} and its proof.

\\

This paper is not the first to extend the Aschbacher--O'Nan--Scott Theorem to specific classes of infinite groups.
In \cite{macpherson_pillay:morley_rank_onan_scott}, a version of the Aschbacher--O'Nan--Scott Theorem for definably primitive permutation groups of finite Morley rank is proved,
while \cite{gelander_and_glasner:countable_primitive_groups} and \cite{gelander_and_glasner:onan_scott} contain Aschbacher--O'Nan--Scott-style classifications for several classes of countably infinite primitive permutation groups in geometric settings.
Macpherson and Praeger (\cite{infinitary_versions}) extend the Aschbacher--O'Nan--Scott Theorem to infinite primitive permutation groups acting on countably infinite sets that posses a minimal closed (in the topology of pointwise convergence) normal subgroup which itself has a minimal closed normal subgroup.
Much of our proof of Theorem~\ref{thm:classification_of_finite_stabilizer} involves establishing that an infinite primitive permutation group with a finite point stabilizer has a unique minimal normal subgroup which itself has a minimal normal subgroup. Having shown this, we could then apply \cite[Theorem 1.1]{infinitary_versions} to obtain a classification. However, the result one obtains with this approach is weaker than Theorem~\ref{thm:classification_of_finite_stabilizer} above, because of the more general situation that is considered in \cite{infinitary_versions}.

Of course it must also be mentioned that much of the strength of original Aschbacher--O'Nan--Scott Theorem derives from the classification of the finite simple groups; nothing similar exists for finitely generated simple groups.

%
% --- The socle
%
\section{The socle}
\label{section:the_socle_is_a_minimal_normal_subgroup}

We use $\sym(\Omega)$ to denote the full symmetric group of a set $\Omega$ and will write $K^g$ to mean $g^{-1} K g$ throughout. In fact, we will always write group actions in this way, with $\alpha^g$ representing the image of $\alpha \in \Omega$ under $g \in G$ whenever $G$ acts on $\Omega$. The orbits of a normal subgroup of a group $G\leq \sym(\Omega)$ form a $G$-invariant partition of $\Omega$; thus, every non-trivial normal subgroup  of a primitive permutation group $G$ acts transitively on $\Omega$. A transitive group $G$ is primitive if and only if for all $\alpha \in \Omega$ the point stabilizer $G_\alpha$ is a maximal subgroup of $G$. Note that if $G$ is transitive and $G_\alpha$ is finite for some $\alpha \in \Omega$, then $G_\alpha$ is finite for all $\alpha \in \Omega$. We say a group $G$ is {\em almost simple} if there exists a normal nonabelian simple subgroup $N$ such that $N \leq G \leq \aut(N)$.

A {\em minimal normal subgroup} of a non-trivial group $G$ is a non-trivial normal subgroup of $G$ that does not properly contain any other non-trivial normal subgroup of $G$. The {\em socle} of $G$, denoted by $\soc(G)$, is the subgroup generated by the set of all minimal normal subgroups of $G$. If $G$ has no minimal normal subgroup then $\soc(G)$ is taken to be $\langle 1 \rangle$. Non-trivial finite permutation groups always have a minimal normal subgroup, so their socle is non-trivial; this is not true in general of infinite permutation groups.

\\

In this section we show that every infinite primitive permutation group $G$ with finite point stabilizers has a unique minimal normal subgroup $M$, and $M$ is characteristically simple, finitely generated, and of finite index in $G$. Furthermore, $M$ is equal to the direct product of finitely many of its simple subgroups. Since $M$ is the only minimal normal subgroup of $G$, it is necessarily equal to the socle of $G$.

\begin{lemma} \label{lemma:all_normal_subgroups_hv_finite_index} If $G \leq \sym(\Omega)$ is primitive, then $|G:N| \leq |G_\alpha|$ for all non-trivial normal subgroups $N \lhd G$ and all $\alpha \in \Omega$.
\end{lemma}
\begin{proof} Let $G\leq \sym(\Omega)$ be primitive. If $N \lhd G$ is non-trivial, then it is transitive, and so for all $\alpha \in \Omega$ we have $G = NG_\alpha$. But this implies that $G_\alpha$ contains a right transversal of $N$ in $G$, and therefore $|G : N| \leq |G_\alpha|$.
\end{proof}

If $G$ is a primitive group of permutations of a set $\Omega$, and some point stabilizer $G_\alpha$ of $G$ is finite, it is simple to see that $G$ is finitely generated by elements of finite order: the group $G_\alpha$ is a maximal subgroup of $G$, and for $\beta \in \Omega \setminus \{\alpha\}$ we have $\langle G_\alpha, G_\beta \rangle = G$ or $G_\alpha$. Because $G$ is primitive, the latter occurs only if $G$ is regular, but infinite primitive permutation groups are never regular.
Hence $G$ is countable (or finite). Since $G$ is transitive, $\Omega$ is also countable.

\begin{thm} \label{socle_size} If $G \leq \sym (\Omega)$ is primitive and infinite, and $G_\alpha$ is finite for some $\alpha \in \Omega$, then $G$ has a unique minimal normal subgroup $M$, and
\begin{enumerate}
\item \label{item:M_bounded} $|G:M| \leq |G_\alpha|$; and
\item \label{item:M_finitely_generated} $M$ is finitely generated and characteristically simple.
\end{enumerate}
\end{thm}

\begin{proof}
Fix $\alpha \in \Omega$. By Lemma~\ref{lemma:all_normal_subgroups_hv_finite_index}, all non-trivial normal subgroups of $G$ are of index at most $|G_\alpha|$. Thus, we may choose a non-trivial normal subgroup $M$ of $G$ whose index in $G$ is greatest.
If $N$ is any non-trivial normal subgroup of $G$, then $M \cap N$ has finite index in $G$ because $|G:N|$ and $|G:M|$ are finite.
Therefore, $M \cap N$ is a non-trivial normal subgroup of $G$ satisfying $|G: M \cap N| \geq |G:M|$, and so $|G: M \cap N| = |G:M|$ by our choice of $M$. Hence $M \cap N = M$. Thus $M$ is the unique minimal normal subgroup of $G$.

Since $G$ is finitely generated and $M$ has finite index in $G$, it follows from Schreier's Lemma (see \cite[Theorem 1.12]{cameron:permutation_groups} for example) that M must also be finitely generated.
Furthermore, a characteristic subgroup of $M$ is normal in $G$, so $M$ must be characteristically simple.
\end{proof}

\begin{thm} \label{thm:socle_is_K_T}
Let $G \leq \sym (\Omega)$ be primitive and infinite, with a finite point stabilizer $G_\alpha$, and let $M$ be its unique minimal normal subgroup. Then $M$ is a direct product
\[M = \prod (K^{g} : {g \in T}), \]
where $K\leq M$ is some infinite nonabelian finitely generated simple group and $T$
is any right transversal of the normaliser $N_G(K)$ in $G$. Furthermore, $G_\alpha$ acts transitively on the components $\{K^g : g \in T\}$ of $M$ by conjugation.
\end{thm}

\begin{proof} The group $M$ is finitely generated, so it contains a maximal proper normal subgroup $N$ (which may be trivial). Since $G = MG_\alpha$, the number of $G$-conjugates of $N$ is equal to the number of $G_\alpha$-conjugates of $N$, which is finite. Let $\{N_1, \ldots, N_\ell\}$ be the set of all $G$-conjugates of $N$, and for $i = 1, \ldots, \ell$, write $K_i := M / N_i$. The groups $K_1, \ldots, K_\ell$ are pairwise isomorphic and simple. The natural homomorphism $\psi: M \rightarrow K_1 \times \cdots \times K_\ell$ given by $h \mapsto (N_1 h, \ldots, N_\ell h)$ is injective, because $\cap_{i=1}^\ell N_i$ is normal in $G$ and is a proper subgroup of $M$. The projection of $\psi(M)$ onto each component $K_i$ is clearly surjective, so $\psi(M)$ is a subdirect product of $K_1 \times \cdots \times K_\ell$.
Since $M$ is infinite the groups $K_1, \ldots, K_\ell$ are infinite, and therefore nonabelian. 
A subdirect product of finitely many nonabelian simple groups is isomorphic to a direct product of some of them (this is sometimes known as Scott's Lemma, see \cite{Scott:Onan_Scott}, but is in fact a special case of a theorem of universal algebra). Thus, $M \cong K_1 \times \cdots \times K_m$ for some $m \leq \ell$. 

Identifying $K_i$ with its isomorphic image in $M$ for $1 \leq i \leq m$, we have $M = K_1 \times \cdots \times K_m$. Since they are simple, the subgroups $K_1, \ldots, K_m$ are the only minimal normal subgroups of $M$ and therefore conjugation by elements of $G$ permutes them amongst themselves. If $\{K_i : i \in I\}$ is any orbit of $G_\alpha$ on $\{K_i : 1 \leq i \leq m\}$, then $\prod_{i \in I} K_i$ is a normal subgroup of $G$, and therefore must equal $M$.  Therefore $G_\alpha$ acts transitively on $\{K_1, \ldots, K_m\}$ by conjugation, and if $T$ is any right transversal of $N_G(K_1)$ in $G$, then $\{K_1, \ldots, K_m\} = \{K_1^g : g \in T\}$.

Finally we note that since $M$ is finitely generated, $K_1$ is finitely generated and because $K_1$ is nonabelian, $M$ is nonabelian.
\end{proof}

%
% --- Proof of Theorem
%
\section{Proof of Theorem~\ref{thm:classification_of_finite_stabilizer}}

The following lemma applies to all permutation groups $G \leq \sym(\Omega)$, not just to those which are primitive. It is a technical lemma that complements Theorem~\ref{thm:classification_of_finite_stabilizer} by describing an explicit permutation embedding of a group of type (\ref{classification_of_finite_stabilizer:wreath_product_case}) into a wreath product in product action.

\begin{lemma} \label{lem:product} Suppose $M = K_1 \times \cdots \times K_m \unlhd G \leq \sym(\Omega)$, such that $M$ is transitive and some point stabiliser $G_\alpha$ transitively permutes the components $\{K_1, \ldots, K_m\}$ of $M$ by conjugation. Suppose further that $M_\alpha = \pi_1(M_\alpha) \times \cdots \times \pi_m(M_\alpha)$, where each $\pi_i$ is the projection of $M$ onto $K_i$. If $K:=K_1$ acts faithfully and transitively on a set $\Gamma$ such that $\pi_1(M_\alpha) = K_\gamma$ for some $\gamma \in \Gamma$, then there exists a homomorphism $\psi : N_{G_\alpha}(K) \rightarrow N_{\sym(\Gamma)_\gamma}(K)$ and a permutational embedding $(\hat{\phi}, \theta)$ of $G$ into $\sym(\Gamma^m)$, where $\theta: \Omega \rightarrow \Gamma^m$ is a bijection such that $\theta(\alpha) = (\gamma, \ldots, \gamma)$, and $\hat{\phi}: G \rightarrow \sym(\Gamma^m)$ is a monomorphism such that $\hat{\phi}(M) = K^m$ and $\hat{\phi}(G_\alpha) \leq \psi(N_{G_\alpha}(K)) \Wr S_m$ acting with its product action on $\Gamma^m$. \end{lemma}

\begin{proof} Choose a set $T:=\{g_1, \ldots, g_m\} \subseteq G_\alpha$ such that $K_i = K^{g_i}$ for $i = 1, \ldots, m$, with $g_1 = 1$. The action of any $g \in G_\alpha$ on $\{K_1, \ldots, K_m\}$ induces a permutation $\sigma(g) \in \sym(m)$ so that $K_i^g = K_{i^{\sigma(g)}}$ for all $i$. Notice that for all $i \in \{1, \ldots, m\}$ and $g \in G_\alpha$ the element $h_i := g_i gg_{i^{\sigma(g)}}^{-1}$ lies in $N:=N_{G_\alpha}(T)$, so for all $x = (x_1, \ldots, x_m) \in M$ we have,
\begin{equation} \label{eq:projection}
 	\pi_i(x)^g = x_i^g = \left (x_i^{g_i^{-1} h_i} \right )^{g_{i^{\sigma(g)}}} = \pi_{i^{\sigma(g)}}(x^g).
\end{equation}
We consider $K^m \leq \sym(\Gamma^m)$ via the product action. Let $\underline{\gamma}:= (\gamma, \ldots, \gamma) \in \Gamma^m$. Now $N$ normalises $K$ and $M_\alpha$, so it normalises $K_\gamma = \pi_1(M_\alpha)$. For $h \in N$ let $\psi(h)$ be the map which sends $\gamma^k$ to $\gamma^{h^{-1}k h}$ for all $k \in K$. It is not difficult to check that $\psi(h)$ is a permutation of $\Gamma$. The map $\psi: N \rightarrow \sym(\Gamma)_\gamma$ is a homomorphism. By considering the action of $K$ on $\Gamma$, we see that conjugation by $h \in N$ and by $\psi(h)$ induce the same automorphism of $K$.

We begin by constructing a permutation isomorphism $(\phi, \theta)$ between $M$ and $K^m$. Define $\phi : M \rightarrow K^m$ by $\phi(x) := (\pi_1(x)^{g^{-1}_1}, \ldots, \pi_m(x)^{g^{-1}_m})$. One may easily verify that $\phi$ is an isomorphism between $M$ and $K^m$. Moreover, for each $i$ we have $\pi_i(M_\alpha)^{g^{-1}_i} = \pi_1(M_\alpha)$, and so $\phi(M_\alpha) = (K_\gamma)^m$. Let $\theta: \Omega \rightarrow \Gamma^m$ be the map with $\theta(\alpha^x) := \underline{\gamma}^{\phi(x)}$ for all $x \in M$. This map is well-defined, because if $\alpha^x = \alpha^y$ for some $x, y \in M$, then $xy^{-1} \in M_\alpha$, and so $\phi(xy^{-1})$ fixes $\underline{\gamma}$. Hence $\theta(\alpha^x) = \underline{\gamma}^{\phi(x)} = \underline{\gamma}^{\phi(y)} = \theta(\alpha^y)$. It is now easily checked that the pair $(\phi, \theta)$ is indeed a permutation isomorphism.

We now extend $(\phi, \theta)$ to a permutation isomorphism $(\hat{\phi}, \theta)$ between $\sym(\Omega)$ and $\sym(\Gamma^m)$ in the following way. Given $f \in \sym(\Omega)$, let $\hat{\phi}(f)$ be the permutation of $\Gamma^m$ which maps each $\underline{\delta} \in \Gamma^m$ to $\theta(\theta^{-1}(\underline{\delta})^f)$. It is easily seen that the map $\hat{\phi} : \sym(\Omega) \rightarrow \sym(\Gamma^m)$ is an isomorphism, and $(\hat{\phi}, \theta)$ is a permutation isomorphism from $\sym(\Omega)$ to $\sym(\Gamma^m)$. Moreover, if $\underline{\delta} \in \Gamma^m$ and $x \in M$ then
write $\beta := \theta^{-1}(\underline{\delta}) \in \Omega$ and observe that $\underline{\delta}^{\hat{\phi}(x)} = \theta(\theta^{-1}(\underline{\delta})^x) = \theta(\beta^x)  = \theta(\beta)^{\phi(x)} = \underline{\delta}^{\phi(x)}$. Hence the restriction of $\hat{\phi}$ to $M$ is $\phi$, and $\hat{\phi}(M) = K^m$.

It remains to show that $\hat{\phi}(G_\alpha) \leq \psi(N) \Wr S_m$, where $\psi(N) \Wr S_m$ is a subgroup of $\sym(\Gamma^m)$ via its product action. Fix $g \in G_\alpha$ and write $\sigma := \sigma(g)$. For $i = 1, \ldots, m$ define $h_{i} := g_i g g_{i^\sigma}^{-1}$. Note that each $h_i$ lies in $N$. Let $z := (\psi(h_1), \ldots, \psi(h_m))\sigma \in \psi(N) \Wr S_m$. We will show that $\hat{\phi}(g) = z$, from which our lemma follows. Indeed, suppose we are given $\underline{k} = (k_1, \ldots, k_m) \in K^m$. Write $x:=\phi^{-1}(\underline{t}) \in M$, so $\pi_i(x) = t_i^{g_i}$. Since the restriction of $\hat{\phi}$ to $M$ is $\phi$,
\begin{equation} \label{eq:final}
\underline{k}^{\hat{\phi}(g)} = 
\hat{\phi}(\hat{\phi}^{-1}(\underline{k}))^{\hat{\phi}(g)} =
\hat{\phi}(\phi^{-1}(\underline{k}))^{\hat{\phi}(g)} =
\hat{\phi}(x)^{\hat{\phi}(g)} =
\hat{\phi}(x^g) =
\phi(x^g).
\end{equation}
Hence for all $i$ we have $\pi_{i^\sigma}(\underline{k}^{\hat{\phi}(g)}) = \pi_{i^\sigma}(\phi(x^g))$, which by the definition of $\phi$ is equal to  $\pi_{i^\sigma}(x^g)^{g_{i^\sigma}^{-1}}$; by (\ref{eq:projection}) this equals $\pi_i (x)^{g g_{i^{\sigma}}^{-1}}$. Thus  $\pi_{i^\sigma}(\underline{k}^{\hat{\phi}(g)}) = k_i^{h_i}$. Now $\pi_{i^\sigma}(\underline{k}^z) = k_i^{\psi(h_i)}$, and since conjugation by $\psi(h_i)$ and $h_i$ induce the same automorphism of $T$, it follows that $\pi_{i^\sigma}(\underline{k}^z) = k_i^{h_i}$. Therefore $\underline{k}^{\hat{\phi}(g)} = \underline{k}^z$. Since $\underline{k}$ was chosen arbitrarily from $K^{m}$, it must be that $\hat{\phi}(g) z^{-1} \in C_{\sym(\Gamma^m)}(K^{m})$. But the full centraliser of $K^{m}$ is semi-regular because $K^{m}$ is transitive (see \cite[Theorem 4.2A]{dixon&mortimer} for example) and $\hat{\phi}(g)$ and $z$ fix $\underline{\gamma}$, so $\hat{\phi}(g) z^{-1}$ must be trivial. \end{proof}

Henceforth, $G \leq \sym (\Omega)$ will be an infinite primitive permutation group with a finite point stabilizer $G_\alpha$, and $M$ the unique minimal normal subgroup of $G$.
By Theorem~\ref{thm:socle_is_K_T}, $M=K_1 \times \cdots \times K_m$, with $K_i \cong K_1$ for $1 \leq i \leq m$, for some infinite nonabelian finitely generated simple group $K:=K_1$ and some finite $m \geq 1$. Let $T \subseteq G_\alpha$ be a right transversal $\{g_1, \ldots, g_m\}$ of $N_{G}(K)$ in $G$, with $g_1 = 1$, whose elements are labelled in such a way that $K_i = K^{g_i}$ for $1 \leq i \leq m$. Recall that $M$ is nonabelian.

We  examine separately the cases where $M$ is simple and acts regularly on $\Omega$, where $M$ is simple and non-regular, and where $M$ is non-simple; these cases will correspond respectively to $G$ being of type (\ref{classification_of_finite_stabilizer:split_extension_case}), (\ref{classification_of_finite_stabilizer:almost_simple_case}) or (\ref{classification_of_finite_stabilizer:wreath_product_case}) in
Theorem~\ref{thm:classification_of_finite_stabilizer}. Since the descriptions of $M$ are mutually exclusive, the same is true of the cases described in the theorem. Theorem~\ref{thm:classification_of_finite_stabilizer} follows immediately from Theorem~\ref{thm:socle_is_K_T},
Proposition~\ref{prop:split_extension},
Remark~\ref{rmk:almost_simple} and Theorem~\ref{thm:wreath_product_type}, given below.

\begin{prop} \label{prop:split_extension}
If $M$ acts regularly on $\Omega$, then $G$ is equal to the split extension $M.G_\alpha$. Furthermore, no non-identity element of $G_\alpha$ induces an inner automorphism of $M$.
\end{prop}

\begin{proof}
Since $G_\alpha \cap M = M_\alpha = \langle 1 \rangle$ the extension $G = M.G_\alpha$ splits. Moreover, $C_G(M)$ is normal in $N_G(M) = G$ and $M$ is the unique minimal normal subgroup of $G$, so either $C_G(M)$ contains $M$ or it is trivial. Since $M$ is nonabelian, $C_G(M)$ must be trivial, and no non-identity element of $G_\alpha$ induces an inner automorphism of $M$.
\end{proof}

In this case, we can identify $\Omega$ with $M$, and the natural action of $G_\alpha$ on $\Omega$ is permutation equivalent to the conjugation action of $G_\alpha$ on $M$. Primitive permutation groups with this structure have the following well-known characterisation.

\begin{prop} \cite[Exercise 2.5.8]{dixon&mortimer} \label{prop:G_prim_iff_Ga_no_normalising} If $H \leq \sym(\Omega)$ and $N \lhd H$ acts regularly on $\Omega$ and $\alpha \in \Omega$, then $H$ is primitive on $\Omega$ if and only if no non-trivial proper subgroup of $N$ is normalised by $H_\alpha$. \qed
\end{prop}

Peter Neumann has pointed out that under the conditions of Proposition~\ref{prop:split_extension} we can say a little more: there are no non-trivial proper $N_{G_\alpha}(K_1)$-invariant subgroups of $K_1$. Indeed, suppose $Y_1 < K_1$ is such a group, and let $Y_i := Y_1^{g_i}$. Define $N:=Y_1 \times \cdots \times Y_m$, a non-trivial proper subgroup of $M$.
Since $G_\alpha$ permutes elements of the set $\{K_i: 1 \leq i \leq m\}$ transitively by conjugation, for each $g \in G_\alpha$ there exists a permutation $\sigma \in S_m$ such that for all integers $i$ satisfying $1 \leq i \leq m$ we have $g_i g = h g_{i^\sigma}$ for some $h \in N_{G_\alpha}(K_1)$. Hence $Y_i^g = Y_1^{g_i g} = Y_1^{h g_{i^\sigma}} = Y_{i^\sigma}$. Thus $G_\alpha$ normalises $N$ and so $G$ is not primitive by Proposition~\ref{prop:G_prim_iff_Ga_no_normalising}.
 
\\

We now turn our attention to the case when $M$ does not act regularly on $\Omega$.

\begin{remark} \label{rmk:almost_simple} 
The group $G$ acts on $M$ by conjugation. Since  $C_G(M)$ is trivial, this action is faithful, and so if $M$ is simple then
\[ M \leq G \leq \aut(M),\]
where $M$ is a finitely generated, simple and nonabelian group of finite index in $G$. In particular, $G$ is almost simple.
\end{remark}

\begin{thm} \label{thm:wreath_product_type} Suppose $M$ is not simple. Then $G$ is permutation isomorphic to a subgroup of $H \Wr_\Delta \sym(\Delta)$, acting via the product action on $\Gamma^m$, where $\Delta = \{1, \ldots, m\}$ and $H \leq \sym(\Gamma)$ is an infinite primitive group with finite point stabilizers, the minimal normal subgroup of which is $K:=K_1$. \end{thm}

\begin{proof}
We have $m > 1$.
Let $R_i:= \pi_i(M_\alpha)$ and $R := R_1 \times \cdots \times R_m \leq M$. Clearly, $M_\alpha \leq R$. Since $G_\alpha$ normalizes $R$, we have $G_\alpha R \leq G$. Now $G_\alpha$ is maximal (because $G$ is primitive) and $G_\alpha R$ is finite. Hence $R \leq G_\alpha$, and from this it follows immediately that $R = M_\alpha$. Let $\Gamma$ be the set of right cosets of $R_1$ in $K$, with $\gamma$ representing the coset $R_1$. The simple group $K$ has a faithful, transitive action on $\Gamma$ by right multiplication, with $K_\gamma = R_1$.

Applying Lemma~\ref{lem:product}, we have that there exists a permutational embedding $(\hat{\phi}, \theta)$ of $G$ into $H \Wr \sym(\Gamma^m)$, where $H = K \psi(N_{G_\alpha}(K))$ for some homomorphism $\psi: N_{G_\alpha}(K) \rightarrow N_{\sym(\Gamma)_\gamma}(K)$. Hence $H_\gamma = K_\gamma \psi(N_{G_\alpha}(K))$, and so all point stabilisers in $H$ are finite.
Since $G$ is primitive on $\Omega$, its image $\hat{\phi}(G)$ is primitive on $\Gamma^m$, and so  $H \Wr S_m$ is primitive. It is well-known (see Lemma~\ref{lem:dixon_mortimer}) that $H$ must be primitive and not regular on $\Gamma$.  By Theorem~\ref{socle_size}, $H$ has a minimal normal subgroup which is unique, and this must obviously be $K$.
\end{proof}

Note that if $M$ is not regular, then $R$ is non-trivial and therefore $K$ is not regular; hence $H$ is of type~(\ref{classification_of_finite_stabilizer:almost_simple_case}). On the other hand, if $M$ is regular then $K$ is regular and hence $H$ is of type~(\ref{classification_of_finite_stabilizer:split_extension_case}).
This concludes our proof of Theorem~\ref{thm:classification_of_finite_stabilizer}. \\

We remark briefly why types of groups present in the classification in \cite{infinitary_versions} do not occur in Theorem~\ref{thm:classification_of_finite_stabilizer}. Groups of affine type (Type 1 in \cite{infinitary_versions}) require $K$ to be finite. Groups with a simple diagonal action (Type 4(a) in \cite{infinitary_versions}) require $K$ to be isomorphic to a point stabilizer in $\soc(G)$, impossible in our case because any stabilizer is finite. Groups satisfying $\soc(G) = M$ with a product action in which $H$ is diagonal (Type 4(b)(ii) in \cite{infinitary_versions}) appear when the possibility that $R_1^g = K^g$ for all $g \in T$ is explored (\cite[pp. 531]{infinitary_versions}), but this cannot occur in our case because $K$ is infinite and $R_1$ is finite.
Finally we repeat the remark in \cite[pp. 534]{infinitary_versions}, that the situation which leads to the twisted wreath product case in \cite{liebeck&praeger&saxl:finite_onan_scott} occurs only if $T$ is infinite. Since $T$ is always finite, twisted wreath products do not appear in our classification.\\

Finally, we note that conditions guaranteeing the primitivity of full wreath products acting in product action are known for both unrestricted wreath products (\cite[Lemma 2.7A]{dixon&mortimer}) and restricted wreath products (\cite[Lemma 3.1]{infinitary_versions}). Of course in class (\ref{classification_of_finite_stabilizer:wreath_product_case}) of Theorem~\ref{thm:classification_of_finite_stabilizer}, no distinction between the restricted and unrestricted wreath product need be made.

\begin{lemma} [{\cite[Lemma 2.7A]{dixon&mortimer}}] \label{lem:dixon_mortimer}
Suppose that $H$ and $S$ are non-trivial groups acting on the sets $\Gamma$ and $\Delta$, respectively. The (unrestricted) wreath product $H \Wr_\Delta S$ is primitive in its product action on $\Gamma^{|\Delta|}$ if and only if:
\begin{enumerate}
\item
	$H$ acts primitively but not regularly on $\Gamma$; and
\item
	$\Delta$ is finite and $S$ acts transitively on $\Delta$. \qed
\end{enumerate}
\end{lemma}

\noindent Notice that in class (\ref{classification_of_finite_stabilizer:wreath_product_case}) of Theorem~\ref{thm:classification_of_finite_stabilizer},  $\Delta$ is always finite.

%
% --- Examples
%
\section{Examples}
\label{section:examples}

Given a prime $p>10^{75}$, there exists a group $T_p$, often called a Tarski--Ol'Shanski{\u\i} Monster, such that every proper non-trivial subgroup of $T_p$ has order $p$ ( \cite[Theorem 28.1]{olshanski}).
Any group $T_p$ can be considered to be an infinite primitive permutation group with finite point stabilizers of type (\ref{classification_of_finite_stabilizer:almost_simple_case}). Indeed, let $H$ be a proper non-trivial subgroup of $T_p$, and let $\Gamma$ be the set $(T_p : H)$ of right cosets of $H$ in $T_p$. Then $T_p$ acts transitively and faithfully, but not regularly, on $\Gamma$ by right multiplication. Any point stabilizer is isomorphic to $H$, a finite and maximal subgroup of $T_p$.

\\

Constructing an example of type (\ref{classification_of_finite_stabilizer:wreath_product_case}) is trivial: let $(T_p; \Gamma)$ be as above, and let $\Delta = \{1, \ldots, m\}$ for some $m \geq 2$. The wreath product $T_p \Wr_\Gamma \sym(\Delta)$ acting on $\Gamma^m$ via the product action is primitive, and the point stabilizers are all finite.

\\

The following result of V.~N.~Obraztsov can be used to construct examples of type (\ref{classification_of_finite_stabilizer:split_extension_case}). In what follows, an automorphism $\sigma$ of a group $G$ is {\em regular} if $g^\sigma \not = g$ for all $g \in G \setminus \langle 1 \rangle$.

\begin{thm}[{\cite[Theorem C, abridged]{obraztsov}}] \label{ob} Let $\{G_i\}_{i \in I}$ be a countable set of non-trivial countable groups containing either three groups or two groups of which one has order at least 3, and let $H$ be an arbitrary countable group. There exists a group $G$ which contains an infinite simple normal subgroup $L$ and satisfies:
\begin{enumerate}
\item \label{item:ob1}
	$G$ is the semi-direct product of $H$ and $L$;
\item \label{item:ob2}
	for each $h \in H \setminus \langle 1 \rangle$, $h$ is a regular automorphism of $L$;
\item \label{item:ob3}
	every proper subgroup of $L$ is either infinite cyclic, or infinite dihedral (if one of the groups $G_i$, $i \in I$, or $H$ has involutions), or contained in a subgroup conjugate in $G$ to some $G_i$, $i \in I$. \qed
\end{enumerate}
\end{thm}

For example, choose distinct odd primes $p_1, p_2, q$ such that $q \not \big| (p_1 -1)$ and $q \not \big| (p_2 -1)$, and take $G_1 := C_{p_1}, G_2 := C_{p_2}$, and $H := C_q$. By Theorem~\ref{ob}, we obtain a group $G$ and an infinite simple normal subgroup $L$ which satisfies (i)--(iii) above. Let $\Omega$  be the set of right cosets of $H$ in $G$ and let $\alpha$ be the coset $H.1 \in \Omega$. The kernel of the action of $G$ on $\Omega$ is contained in $G_\alpha = H = C_q$. If $G$ is not faithful then $H \unlhd G$, and so $G$ (being the semi-direct product of $H$ and $L$) is in fact the direct product of $H$ and $L$; this is impossible, because every non-trivial element in $H$ is a regular automorphism of $L$. Hence $G$ acts faithfully on $\Omega$, and we may consider $G$ to be a transitive subgroup of $\sym(\Omega)$.

Clearly $L \unlhd G$ is regular on $\Omega$. We claim that $G_\alpha$ normalizes no proper non-trivial subgroup of $L$. From this it will follow from Proposition~\ref{prop:G_prim_iff_Ga_no_normalising} that $G$ is primitive.
Since $G_\alpha = H$ there exists $h \in G_\alpha \setminus \langle 1 \rangle$. Let $K$ be a proper non-trivial subgroup of $L$. Then $K$ is isomorphic to $C_{\infty}, C_{p_1}$ or $C_{p_2}$. Suppose, for a contradiction, that $K^h = K$. Since $h$ has order $q$, it is easy to see that $K$ cannot be infinite and cyclic, so it is isomorphic to $C_p$ for some $p \in \{p_1, p_2\}$. Fix $k \in K \setminus \langle 1 \rangle$. Then $k^h \not = 1$ and there exists $i \in \Z$ such that $k^h = k^i$.
Now $p \big| (i^p - i)$, so $k^{i^p} = k^i$. But $k^{h^{p}} = k^{i^p}$, so $k^i$ is fixed by $h^{p-1}$. Since $h^{p-1} \in H \setminus \langle 1 \rangle$ is a regular automorphism of $L$, and $k^i$ is non-trivial, we have a contradiction.
Hence  $G_\alpha$ normalizes no proper non-trivial subgroup of $L$, and it follows that $G$ is an infinite primitive permutation group, with a finite stabilizer, of type (\ref{classification_of_finite_stabilizer:split_extension_case}).

%
% --- Acknolwedgements
%

\\

Some of the work for this paper took place while the author was a Philip T.~Church Postdoctoral Fellow at Syracuse University. The author would like to thank Mark Sapir for pointing out that \cite{obraztsov} could be used to construct an example of type (\ref{classification_of_finite_stabilizer:split_extension_case}), and Peter Neumann, Cheryl Praeger and Luke Morgan, for carefully reading early drafts of this paper. The author would also like to thank the anonymous referee, whose suggestions greatly improved this work.

%
% --- Bibliography
%

\vspace{5mm}

\end{document}